\documentclass[3p,nopreprintline]{elsarticle}

\journal{}

\usepackage{amsthm}
\usepackage{mathrsfs}
\usepackage{array}
\usepackage{amsmath}
\usepackage{amsfonts}
\usepackage{amssymb}
\usepackage{float}
\usepackage{latexsym}
\usepackage{graphicx}
\usepackage{color}
\usepackage{booktabs}
\usepackage{enumitem}
\usepackage[T1]{fontenc}
\usepackage{caption}
\usepackage[
  colorlinks=true,
  linkcolor=blue,
  citecolor=blue,
  urlcolor=blue,
  pdfstartview=FitH
]{hyperref}
\usepackage{doi}
\graphicspath{{figures/}}
\newtheorem{theorem}{Theorem}[section]
\newtheorem{lemma}[theorem]{Lemma}

\newtheorem{corollary}[theorem]{Corollary}

\newtheorem{problem}[theorem]{Problem}
\newtheorem{conjecture}[theorem]{Conjecture}
\theoremstyle{definition}

\newcommand{\diam}{\operatorname{diam}}

\newcommand{\calT}{\mathcal T}

\begin{document}

\begin{frontmatter}

\title{{\bf The number of Laplacian eigenvalues of trees less than one}}
\author[a]{Jiaxin Guo$^{*\,}$}
\author[b]{Tao Hu}
\author[c]{Quanyu Tang}

\address[a]{Information Engineering University, 450001 Zhengzhou, China}
\address[b]{Department of Information, Risk, and Operations Management, McCombs School of Business,
The University of Texas at Austin, Austin, TX 78712, USA}
\address[c]{School of Mathematical Sciences, University of Science and Technology of China, Hefei 230026, P.~R.~China}

\cortext[cor1]{Email addresses: jiaxin\_guo7@126.com (Jiaxin Guo), tao\_hu@utexas.edu (Tao Hu), {tangquanyu827@gmail.com} (Quanyu Tang)}

\begin{abstract}
Let $m_T[0,1)$ denote the number of Laplacian eigenvalues of a tree $T$ that are strictly less than $1$.
Guo, Xue and Liu \cite{GXL} proved that every tree $T$ of diameter $d$ satisfies $m_T[0,1)\ge \lceil(d+1)/3\rceil$, and that this bound is sharp when $d\equiv2\pmod3$.
It is classical that every graph of diameter $d$ has domination number at least $\lceil(d+1)/3\rceil$.
In this paper, for any diameter $d$, $m_T[0,1)=\lceil(d+1)/3\rceil$ if and only if $\gamma(T)=\lceil(d+1)/3\rceil$; we characterize all trees satisfying this necessary and sufficient condition, and also prove that almost all trees satisfy $m_T[0,1)\ge\lceil(d+1)/3\rceil+1$.
\end{abstract}


\end{frontmatter}

\section{Introduction}\label{sec:1}

In this paper, all graphs considered  are finite and simple.
Given a graph $G$ on $n$ vertices, the Laplacian matrix of $G$ is $L(G)=D(G)-A(G)$, where $D(G)$ and $A(G)$ denote the diagonal matrix of vertex degrees and the adjacency matrix of $G$, respectively.
The eigenvalues of $L(G)$ are called the Laplacian eigenvalues of $G$.
Since $L(G)$ is positive semidefinite and singular with the all-ones vector as an eigenvector, we arrange its eigenvalues in nondecreasing order as $0=\mu_1(G)\le \mu_2(G)\le \cdots \le \mu_n(G)$.
For an interval $I\subseteq\mathbb R$, we write $m_G I$ for the number of Laplacian eigenvalues of $G$ lying in $I$, that is, $m_G I=|\{i:\mu_i(G)\in I\}|$.
In particular, $m_G[0,1)=|\{i:\mu_i(G)<1\}|$ is the number of Laplacian eigenvalues of $G$ which are strictly less than $1$.

The distribution of Laplacian eigenvalues can reveal useful structural information of a graph.
Sin~\cite{Sin2020} and Jacobs, Oliveira and Trevisan~\cite{JacobsOliveiraTrevisan2021} independently proved that every tree of order $n$ has at least $\lceil n/2\rceil$ Laplacian eigenvalues less than its average degree $2-2/n$.
Wang, Yan, Fang, Geng and Tian~\cite{Wang2020} proved that a connected graph $G$ of order $n$ satisfies $m_G(n-1,n]\le\kappa(G)$ and $m_G(n-1,n]\le\chi(G)-1$, where $\kappa(G)$ and $\chi(G)$ are its vertex-connectivity and chromatic number.
Ahanjideh, Akbari, Fakharan and Trevisan~\cite{Ahanjideh2022} obtained $m_G(n-\alpha(G),n]\le n-\alpha(G)$, where $\alpha(G)$ is the independence number, and proved that $m_G(n-d+3,n]\le n-d-1$ for every connected graph of diameter $d\ge4$; they also characterized the bipartite graphs satisfying $m_G[0,1)=\alpha(G)$.

For vertices $u,v\in V(G)$, let $d_G(u,v)$ denote their distance.
The diameter of a connected graph $G$ is $\diam(G)=\max_{u,v\in V(G)}d_G(u,v)$, and a path of length $\diam(G)$ is called a \emph{diametral path}.
For convenience, we write $\diam(G)$ as $d(G)$.
A vertex of degree one is a \emph{pendant vertex}.
A set $D\subseteq V(G)$ is a dominating set if every vertex in $V(G)\setminus D$ has a neighbour in $D$.
The domination number $\gamma(G)$ is the minimum cardinality of a dominating set of $G$.

The diameter is a central parameter in the study of Laplacian eigenvalue distribution.
Xu and Zhou~\cite{XuZhou2023} proved that if $G$ is a connected graph of order $n$ and diameter $d\ge2$ with $G\not\cong P_{d+1}$, then $m_G[n-d+2,n]\le n-d$, and they completely characterized the equality cases.
They subsequently proved that $m_G[n-d+1,n]\le n-d+1$ whenever $1\le d\le n-3$~\cite{XuZhou2024}.
When considering the small Laplacian eigenvalues, the classical domination bound gives $\gamma(T)\ge(d+1)/3$ for every tree $T$ of diameter $d$.
Xue, Liu, Yu and Shu~\cite{Xue2020} proved that equality holds if and only if $1$ is a Laplacian eigenvalue of multiplicity $n-d$, and they characterized all trees with this property.
Guo, Xue and Liu~\cite{GXL} subsequently established $m_T[0,1)\ge(d+1)/3$ and characterized all trees attaining this unrounded bound, which is possible only when $d\equiv2\pmod3$.

Related questions have recently been studied beyond trees.
For unicyclic graphs, Moon and Park~\cite{MoonPark2025} proved the lower bound $m_G[0,1)\ge\lceil d/3\rceil+\lceil g/6\rceil-1$ in terms of the diameter $d$ and girth $g$, and determined the exact value for certain lollipop graphs.
Barik and Jena~\cite{BarikJena2026} subsequently characterized the unicyclic graphs satisfying $m_G[0,1)=\gamma(G)$ and also proved $m_G[0,2)\le n-\gamma(G)$ for unicyclic graphs of order $n$, including a characterization of equality.

The special eigenvalue $1$ also admits useful structural reduction methods.
Barik, Lal and Pati~\cite{BLP} described how this multiplicity changes when two trees are joined by an edge, and their theorem is recalled in Section~\ref{sec:2}.
Writing $p(T)$ for the number of pendant vertices and $q(T)$ for the number of vertices adjacent to pendant vertices, Tian, Wang and Song~\cite{TianWangSong2025} introduced a reduced tree $T_{\rm r}$ and proved $m_T(1)=p(T)-q(T)+m_{T_{\rm r}}(1)$; for a reduced tree of order $n\ge6$, they obtained the sharp bound $m_T(1)\le(n-2)/4$ and determined all equality cases.
Tian and Wong~\cite{TianWong2026} extended the reduction formula to general graphs and proved that a reduced tree of order $n\ge6$ without a pendant path $P_3$ satisfies $m_T(1)\le(n-6)/4$, again with a complete characterization of equality.
These results provide a theoretical basis for the subsequent arguments.

Among these parameters, domination is especially closely connected with the eigenvalues less than $1$.
Hedetniemi, Jacobs and Trevisan proved the following fundamental bound.

\begin{theorem}[{\cite{HJT}}]
\label{thm:HJT}
If $G$ is a graph with domination number $\gamma(G)$, then $m_G[0,1)\le \gamma(G)$.
\end{theorem}

Theorem~\ref{thm:HJT} gives an upper bound for $m_G[0,1)$ and, at the same time, a spectral lower bound for $\gamma(G)$.
Together with the diameter lower bound above, it shows that a tree with large diameter must have many Laplacian eigenvalues less than $1$, while their number cannot exceed its domination number.
This inequality is sharp, and equality can be achieved under certain conditions.

The result of Guo, Xue and Liu characterizes equality in the unrounded bound
$m_T[0,1)\ge(d+1)/3$, which can occur only when $d\equiv2\pmod3$.
The equality cases for the full ceiling bound
$m_T[0,1)\ge\lceil(d+1)/3\rceil$ remained open when
$d\not\equiv2\pmod3$.
At the end of their paper, they posed the following problem and conjecture.

\begin{problem}[{\cite{GXL}}]
Characterize all trees $T$ satisfying the equality $m_T[0,1)=\lceil(d(T)+1)/3\rceil$.
\end{problem}

For the conjecture, let $\calT_n$ denote the set of unlabeled trees of order $n$, and define
$$ \calT_n^* := \left\{ T\in\calT_n: m_T[0,1)=\left\lceil\frac{d(T)+1}{3}\right\rceil \right\}.$$

\begin{conjecture}[{\cite{GXL}}]
With the notation above,
$$\lim_{n\to\infty}\frac{|\mathcal T_n^*|}{|\mathcal T_n|}=0.$$
\end{conjecture}

In this paper we solve the problem and prove the conjecture.
To state the main result, let $P=v_0v_1\cdots v_d$ be a diametral path.
Adding pendant vertices at $v_i$ means joining new vertices of degree one to $v_i$.
To attach $K_{1,t}$ at $v_i$, where $t\ge1$, we take a star with center $u$ and add the edge $uv_i$.

\begin{theorem}
\label{thm:classification-intro}
Let $T$ be a tree of diameter $d$.
Then $m_T[0,1)=\lceil(d+1)/3\rceil$ if and only if $T$ has a diametral path $P=v_0v_1\cdots v_d$ such that $T$ belongs to one of the following two classes.

\medskip

\noindent
\textbf{(i) Caterpillars.}
There is a minimum dominating set $D$ of the path $P$ such that $T$ is obtained from $P$ by adding any number of pendant vertices, possibly zero, only at internal vertices $v_i\in D$, and by attaching no other vertices.

\medskip

\noindent
\textbf{(ii) Trees with one attached star.}
Here $d=3k$ for some $k\ge2$.
There is an integer $r$ with $1\le r\le k-1$ and an integer $t\ge1$ such that $T$ is obtained as follows.
First take the path $P=v_0v_1\cdots v_{3k}$.
Take a star $K_{1,t}$ with center $u$ and add the edge $uv_{3r}$.
In addition, add any number of pendant vertices, possibly zero, only at vertices in
$ \{v_1,v_4,\ldots,v_{3r-2}\} \cup \{v_{3r+2},v_{3r+5},\ldots,v_{3k-1}\},$
and attach no other vertices.
\end{theorem}

The description in (i) has the following explicit form in terms of the positions where pendant vertices may be attached.
The set of vertices receiving new pendant vertices may be any subset of one of the following admissible sets of vertices:
\begin{enumerate}[label=\textup{(\alph*)}]
\item if $d=3k+2$, this set is uniquely $v_1,v_4,\ldots,v_{d-1}$;
\item if $d=3k+1\ge4$, it starts at $v_1$, ends at $v_{d-1}$, and its consecutive index gaps consist of one gap of length $2$ and gaps of length $3$ elsewhere;
\item if $d=3k$, it starts at $v_1$, ends at $v_{d-1}$, and its consecutive index gaps consist either of one gap of length $1$ and gaps of length $3$ elsewhere, or of two gaps of length $2$
and gaps of length $3$ elsewhere.
\end{enumerate}
Figure~\ref{fig:admissible-sets} displays the three gap patterns in the last two items.
In this figure, the vertices of a minimum dominating set are written as $D=\{u_1,\ldots,u_q\}$ in their order along $P$ and are colored red, while the remaining vertices of $P$ are white.

\begin{figure}[H]
\centering
\includegraphics[width=0.82\textwidth]{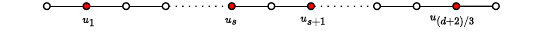}
\par
{\small \textup{(a)} $d\equiv1\pmod3$: one gap of length $2$.}\par
\vspace{20pt}
\includegraphics[width=0.90\textwidth]{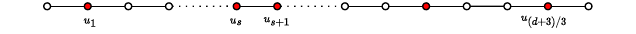}
\par
{\small \textup{(b)} $d\equiv0\pmod3$: one gap of length $1$.}\par
\vspace{20pt}
\includegraphics[width=0.94\textwidth]{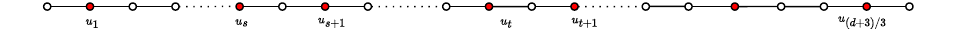}
\par
{\small \textup{(c)} $d\equiv0\pmod3$: two gaps of length $2$, which may be adjacent.}
\caption{The admissible minimum dominating sets governing the attachment positions in class \textup{(i)}.}
\label{fig:admissible-sets}
\end{figure}

The small cases $d=0$ and $d=1$ give the bare paths $P_1$ and $P_2$.
In particular, item~\textup{(a)} recovers the non-ceiling equality case of Guo, Xue and Liu; it is also consistent with the multiplicity formulation for the eigenvalue $1$ obtained by Xue, Liu, Yu and Shu~\cite{Xue2020}.
Figure~\ref{fig:type-I} shows a caterpillar of class \textup{(i)} in the case $d\equiv1\pmod3$, and Figure~\ref{fig:type-II} shows a tree of class \textup{(ii)} in the case $d\equiv0\pmod3$.

\begin{figure}[H]
\centering
\includegraphics[width=0.94\textwidth]{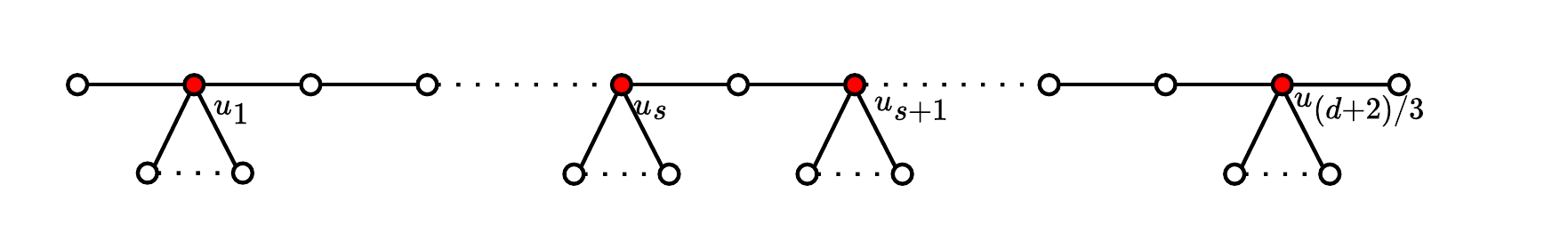}
\caption{A caterpillar in the case $d\equiv1\pmod3$.}
\label{fig:type-I}
\end{figure}

\begin{figure}[H]
\centering
\includegraphics[width=0.94\textwidth]{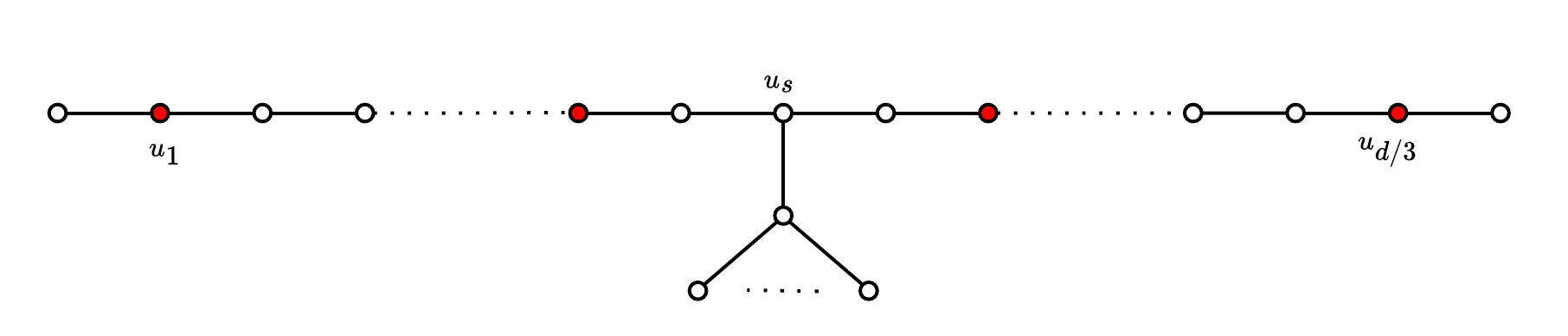}
\caption{The second class for $d\equiv0\pmod3$.}
\label{fig:type-II}
\end{figure}

Our second main result proves the asymptotic conjecture.

\begin{theorem}
\label{thm:asymptotic-intro}
Let $\calT_n$ be the set of unlabeled trees of order $n$, and let
$$ \calT_n^* = \left\{ T\in\calT_n: m_T[0,1)=\left\lceil\frac{d(T)+1}{3}\right\rceil \right\}.$$
Then $\lim_{n\to\infty}\frac{|\calT_n^*|}{|\calT_n|}=0$.
Equivalently, almost all trees satisfy $m_T[0,1)\ge \lceil(d(T)+1)/3\rceil+1$.
\end{theorem}

The paper is organized as follows.
Section~\ref{sec:2} gives the spectral preliminaries and proves that every vertex of a tree attaining equality has distance at most $2$ from a diametral path.
Section~\ref{sec:3} applies these tools to describe and exclude the forbidden structures along a diametral path.
Section~\ref{sec:4} proves Theorem~\ref{thm:classification-intro}.
Section~\ref{sec:5} proves Theorem~\ref{thm:asymptotic-intro} by directly enumerating the two equality families.

\section{Preliminaries}\label{sec:2}

We first record the classical domination lower bound.

\begin{lemma}[{\cite{HHS}}]
\label{lem:dom-lower}
If $T$ is a tree of diameter $d$, then $\gamma(T)\ge \lceil(d+1)/3\rceil$.
\end{lemma}

The following contraction result will be used repeatedly.

\begin{lemma}[{\cite{GMS}}]
\label{lem:contraction}
Let $T$ be a tree on $n$ vertices with Laplacian eigenvalues $\mu_1(T)\le\cdots\le\mu_n(T)$, and let $T'$ be the tree obtained from $T$ by contracting an edge.
Then $T'$ has $n-1$ vertices and
$$ \mu_i(T)\le\mu_i(T')\qquad\text{for }1\le i\le n-1. $$
\end{lemma}

\begin{corollary}
\label{cor:contraction}
Let $T$ be a tree, and suppose that $V(T)=U_1\cup U_2\cup\cdots\cup U_r$ is a partition into nonempty connected subsets.
Let $T_0$ be the graph obtained by contracting each $U_i$ to a single vertex.
Then $T_0$ is a tree and $m_T[0,1)\ge m_{T_0}[0,1)$.
\end{corollary}

\begin{proof}
Since $T$ is connected and each $U_i$ is connected, $T_0$ is connected.
Moreover, two distinct blocks cannot be joined by two different edges of $T$, because those two edges together with the unique paths inside the two blocks would form a cycle.
Similarly, any cycle in $T_0$ would lift to a cycle in $T$.
Hence $T_0$ is a tree.
Each $U_i$ is connected, so it induces a subtree of $T$; contracting the edges of these subtrees one by one produces $T_0$ from $T$.
By Lemma~\ref{lem:contraction}, contracting an edge never decreases any Laplacian eigenvalue, so it cannot increase the number of Laplacian eigenvalues strictly less than $1$.
Therefore $m_{T_0}[0,1)\le m_T[0,1)$.
\end{proof}

\begin{lemma}[{\cite{GXL}}]
\label{lem:spectral-lower}
If $T$ is a tree of diameter $d$, then $m_T[0,1)\ge \lceil(d+1)/3\rceil$.
\end{lemma}

A vertex $v$ of a tree $T$ is called a \emph{nonzero vertex for the eigenvalue $1$} if there is a vector $z$ such that $L(T)z=z$ and $z_v\ne0$.

For a graph $G$, let $m_G(1)$ denote the multiplicity of $1$ as a Laplacian eigenvalue of $G$.

\begin{lemma}[{\cite{BLP}}]
\label{lem:joining-eigenvalue-one}
Let $T_1, T_2$ be two trees and $v_1 \in T_1, v_2 \in T_2$. Form a new tree $T$ by joining the vertices $v_1$, $v_2$ by an edge. Then the following hold:
\begin{enumerate}
\item[(1)] If $v_1$ is a nonzero vertex of $T_1$, then
\[
m_T(1) = m_{T_1}(1) + m_{T_2}(1) - 1.
\]
\item[(2)] If $v_1$ is not a nonzero vertex of $T_1$ and $v_2$ is not a nonzero vertex of $T_2$, then
\[
m_{T_1}(1) + m_{T_2}(1) \le m_T(1) \le m_{T_1}(1) + m_{T_2}(1) + 1.
\]
\end{enumerate}
\end{lemma}

\begin{lemma}
\label{lem:cut-additivity}
Let $T_1$ and $T_2$ be two trees, and let $T$ be obtained by adding an edge $uv$, where $u\in V(T_1)$ and $v\in V(T_2)$.
If $u$ is a nonzero vertex of $T_1$, then
$$ m_T[0,1)=m_{T_1}[0,1)+m_{T_2}[0,1). $$
\end{lemma}

\begin{proof}
Let $T^\ast=T_1\cup T_2$.
Since $u$ is a nonzero vertex of $T_1$, Lemma~\ref{lem:joining-eigenvalue-one} gives
$$ m_T(1)=m_{T_1}(1)+m_{T_2}(1)-1=m_{T^\ast}(1)-1. $$

Now $T^\ast$ has exactly $|V(T^\ast)|-m_{T^\ast}[0,1)-m_{T^\ast}(1)$ Laplacian eigenvalues greater than $1$.
If $b=e_u-e_v$ is the incidence vector of the added edge, then
$$ L(T)=L(T^\ast)+bb^{\top}. $$
The rank-one positive semidefinite interlacing inequalities imply
$$ m_{T^\ast}[0,1)-1\le m_T[0,1)\le m_{T^\ast}[0,1), $$
and they also imply that $T$ has at most $|V(T^\ast)|-m_{T^\ast}[0,1)-m_{T^\ast}(1)+1$ Laplacian eigenvalues greater than $1$.
If $m_T[0,1)=m_{T^\ast}[0,1)-1$, then $m_T(1)=m_{T^\ast}(1)-1$ would force $T$ to have
$$ |V(T)|-\bigl(m_{T^\ast}[0,1)-1\bigr)-\bigl(m_{T^\ast}(1)-1\bigr)=|V(T^\ast)|-m_{T^\ast}[0,1)-m_{T^\ast}(1)+2 $$
Laplacian eigenvalues greater than $1$, a contradiction.
Hence $m_T[0,1)=m_{T^\ast}[0,1)$.
Finally, the Laplacian spectrum of the disjoint union $T^\ast=T_1\cup T_2$ is the multiset union of the spectra of $T_1$ and $T_2$, and thus
$$ m_{T^\ast}[0,1)=m_{T_1}[0,1)+m_{T_2}[0,1). $$
Therefore $m_T[0,1)=m_{T_1}[0,1)+m_{T_2}[0,1)$.
This proves the result.
\end{proof}

The next lemma restricts the distance of vertices from a diametral path.

\begin{lemma}
\label{lem:depth-three}
Let $T$ be a tree of diameter $d$, and let $P=v_0v_1\cdots v_d$ be a diametral path.
If some vertex of $T$ has distance at least $3$ from $P$, then $m_T[0,1) \ge \lceil(d+1)/3\rceil+1$.
\end{lemma}

\begin{proof}
Choose a vertex $t$ whose distance to $P$ is at least $3$, and let $v_j$ be the vertex of $P$ closest to $t$.
Let $v_j-w_1-w_2-w_3$ be the first three edges of the unique path from $v_j$ to $t$, where $w_1,w_2,w_3\notin V(P)$.

Delete the edges of $P$ and also delete $v_jw_1$, $w_1w_2$, and $w_2w_3$.
Let $U_i$ be the component containing $v_i$, and let $W_1,W_2,W_3$ be the components containing $w_1,w_2,w_3$, respectively.
These $d+4$ sets are precisely the components of the resulting forest.
Contracting all these components gives a tree $T^*$ obtained from $P_{d+1}$ by attaching a pendant path of length $3$ at $v_j$.
By Corollary~\ref{cor:contraction}, $m_T[0,1)\ge m_{T^*}[0,1)$.

Let $P_3=w_1w_2w_3$.
The vector $z$ defined by
$$ z(w_1)=1,\qquad z(w_2)=0,\qquad z(w_3)=-1 $$
satisfies $L(P_3)z=z$.
Hence $w_1$ is a nonzero vertex of $P_3$ for the eigenvalue $1$.
The Laplacian spectrum of $P_3$ is $\{0,1,3\}$, and therefore $m_{P_3}[0,1)=1$.

The tree $T^*$ is obtained from the vertex-disjoint trees $P_3$ and $P_{d+1}$ by adding the edge $w_1v_j$.
Lemma~\ref{lem:cut-additivity} gives
$$ m_{T^*}[0,1)=m_{P_3}[0,1)+m_{P_{d+1}}[0,1)=m_{P_{d+1}}[0,1)+1. $$
Since $P_{d+1}$ has diameter $d$, Lemma~\ref{lem:spectral-lower} yields $m_{P_{d+1}}[0,1)\ge \lceil(d+1)/3\rceil$.
Therefore
$$ m_T[0,1)\ge m_{T^*}[0,1)=m_{P_{d+1}}[0,1)+1 \ge \left\lceil\frac{d+1}{3}\right\rceil+1,$$
as required.
\end{proof}

The same argument proves the following stronger fact.
Let $T$ be any tree, let $v\in V(T)$, and let $P_3=w_1w_2w_3$ be vertex-disjoint from $T$.
If $H$ is obtained from $T\cup P_3$ by adding the edge $vw_1$, then
$$ m_H[0,1)=m_T[0,1)+1. $$

\begin{corollary}
\label{cor:distance-two}
If $T$ is a tree of diameter $d$ satisfying $m_T[0,1)=\lceil(d+1)/3\rceil$, then $T$ has a diametral path $P$ such that every vertex of $T$ has distance at most $2$ from $P$.
\end{corollary}

\begin{proof}
Let $P$ be a diametral path.
If some vertex had distance at least $3$ from $P$, Lemma~\ref{lem:depth-three} would give $m_T[0,1) \ge \lceil(d+1)/3\rceil+1$, contrary to equality.
\end{proof}

\section{Main lemmas}\label{sec:3}

Throughout this section, let $P=v_0v_1\cdots v_d$ be a diametral path of a tree $T$.
In the applications below, every vertex of $T$ has distance at most $2$ from $P$, by Corollary~\ref{cor:distance-two}.
Consequently, every component of $T-V(P)$ is either a single pendant vertex adjacent to a vertex of $P$, or a star whose center is adjacent to a vertex of $P$.

We first record the monotonicity needed to pass from a forbidden subtree to the original tree.

\begin{lemma}[{\cite{GXL}}]
\label{lem:pendant-monotonicity}
Let $G'$ be obtained from a graph $G$ by deleting a pendant vertex.
Then $m_G[0,1)\ge m_{G'}[0,1)$.
\end{lemma}

The next facts identify the pieces that occur after the cut edges are deleted.

\begin{lemma}
\label{lem:path-count}
For every $n\ge1$,
$$ m_{P_n}[0,1)=\left\lceil\frac{n}{3}\right\rceil. $$
If $n\equiv0\pmod3$, both end vertices of $P_n$ are nonzero vertices for the eigenvalue $1$.
\end{lemma}

\begin{proof}
The Laplacian eigenvalues of $P_n$ are $2-2\cos(j\pi/n)$, where $j=0,1,\ldots,n-1$.
Such an eigenvalue is less than $1$ precisely when $j<n/3$, which gives the stated count.

Now assume $n\equiv0\pmod3$ and write the path as $v_0v_1\cdots v_{n-1}$.
Starting at $v_0$, assign the values $1,0$ and propagate by $z(v_{i+1})=z(v_i)-z(v_{i-1})$.
The values repeat with pattern $1,0,-1,-1,0,1$.
Since $n\equiv0\pmod3$, the equation at the opposite end is also satisfied.
Thus $L(P_n)z=z$ and $z(v_0)\ne0$.
The other end follows by symmetry.
\end{proof}

\begin{lemma}
\label{lem:caterpillar-equality}
Let $T$ be obtained from a diametral path $P$ of length $d$ by adding pendant vertices only at vertices belonging to a minimum dominating set of $P$.
Then
$$ m_T[0,1)=\left\lceil\frac{d+1}{3}\right\rceil. $$
\end{lemma}

\begin{proof}
The specified minimum dominating set of $P$ also dominates $T$.
Theorem~\ref{thm:HJT}, Lemma~\ref{lem:dom-lower} and Lemma~\ref{lem:spectral-lower} give
$$ \left\lceil\frac{d+1}{3}\right\rceil\le m_T[0,1)\le\gamma(T)\le\left\lceil\frac{d+1}{3}\right\rceil. $$
\end{proof}

\begin{lemma}
\label{lem:nonzero-tail}
Let $T$ be obtained from a path $P=v_0v_1\cdots v_d$ by adding pendant vertices only to vertices of $P$.
Suppose that $v_p$ is the vertex of $P$ with the largest index having a pendant neighbour outside $V(P)$.
If $d-p\equiv1\pmod3$, then $v_d$ is a nonzero vertex of $T$.
Symmetrically, if $v_p$ is the vertex of $P$ with the smallest index having a pendant neighbour outside $V(P)$ and $p\equiv1\pmod3$, then $v_0$ is a nonzero vertex of $T$.
\end{lemma}

\begin{proof}
We prove the first assertion.
Choose a pendant neighbour $w\in V(T)\setminus V(P)$ of $v_p$.
Set $z(v_d)=1$ and $z(v_{d-1})=0$, and propagate from right to left by $z(v_{i-1})=z(v_i)-z(v_{i+1})$ until $v_p$ is reached.
The values repeat with pattern $1,0,-1,-1,0,1$, so $d-p\equiv1\pmod3$ gives $z(v_p)=0$.
Set all coordinates to the left of $v_p$ equal to $0$, put $z(w)=-z(v_{p+1})$, and assign value $0$ to every other pendant vertex outside $V(P)$.
A direct substitution gives $L(T)z=z$ at every vertex, and $z(v_d)=1$.
The second assertion follows by reversing the path.
\end{proof}

\begin{lemma}
\label{lem:wrong-leaf}
Let $d\equiv2\pmod3$, and let $T$ be obtained from a diametral path $P=v_0v_1\cdots v_d$ by adding one pendant vertex at $v_p$, where $1\le p\le d-1$ and $p\not\equiv1\pmod3$.
Then
$$ m_T[0,1)=\frac{d+1}{3}+1. $$
\end{lemma}

\begin{proof}
Suppose first that $p\equiv0\pmod3$.
Delete the edge $v_{p-1}v_p$.
The component containing $v_0$ is a path of diameter $p-1\equiv2\pmod3$, and the component containing $v_d$ is a path of diameter $d-p+1\equiv0\pmod3$, with the added pendant vertex as one end.
The joining end of the component containing $v_0$ is nonzero by Lemma~\ref{lem:path-count}.
Lemmas~\ref{lem:cut-additivity} and \ref{lem:path-count} give
$$ m_T[0,1)=\frac{p}{3}+\left(\frac{d-p+1}{3}+1\right)=\frac{d+1}{3}+1. $$

If $p\equiv2\pmod3$, delete $v_pv_{p+1}$ and apply the same argument from the other end.
\end{proof}

\begin{figure}[H]
\centering
\includegraphics[width=0.55\textwidth]{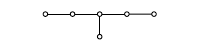}
\caption{The six-vertex tree $T_6$.}
\label{fig:T6}
\end{figure}

Let $T_6$ denote the six-vertex tree shown in Figure~\ref{fig:T6}, whose unique vertex of degree three has incident paths of lengths $1,2,2$.
A direct calculation gives
$$ \det(\lambda I-L(T_6))=\lambda(\lambda-2)(\lambda^2-3\lambda+1)(\lambda^2-5\lambda+3), $$
and consequently
$m_{T_6}[0,1)=3$.

\subsection{The case \texorpdfstring{$d\equiv1\pmod3$}{d congruent to 1 modulo 3}}

Let $d=3k+1$, where $k\ge1$.
For $1\le s\le k$, put

$$ V_s=\{v_1,v_4,\ldots,v_{3s-2}\}\cup\{v_{3s},v_{3s+3},\ldots,v_{3k}\}. $$
The sets $V_s$ are precisely the sets of internal vertices of $P$ at which pendant vertices may be added.
Their consecutive gaps are $3$, except for one gap of length $2$.
Each $V_s$ has $k+1$ vertices and dominates $P$; hence $V_s$ is a minimum dominating set of $P$.

We consider the following four structures, each of which contains the whole path $P$:
\begin{description}
\item[\textup{(F1)}] One pendant vertex is attached to $v_{3a-1}$, where $1\le a\le k$.
\item[\textup{(F2)}] One pendant vertex is attached to each of $v_{3a}$ and $v_{3b+1}$, where $1\le a\le b\le k-1$.
\item[\textup{(F3)}] A pendant path of length two is attached to $v_{3a+1}$, where $1\le a\le k-1$.
\item[\textup{(F4)}] A pendant path of length two is attached to $v_{3a}$, where $1\le a\le k-1$.
\end{description}

\begin{lemma}
\label{lem:F-structure}
Let $T$ have diameter $3k+1$, and suppose every vertex of $T$ has distance at most $2$ from $P$.
Then $T$ is a caterpillar whose pendant vertices are attached only at vertices of some $V_s$ if and only if $T$ contains none of \textup{(F1)--(F4)} as a subtree containing $P$.
\end{lemma}

\begin{proof}
Suppose first that every component outside $P$ is a single pendant vertex, and let $I=\{i: v_i \text{ has an added pendant neighbour}\}$.
The set $\{v_i:i\in I\}$ is contained in some $V_s$ precisely when no $i\in I$ is congruent to $2$ modulo $3$, and every $i\in I$ congruent to $1$ occurs before every $i\in I$ congruent to $0$.
Failure of the first condition gives \textup{(F1)}, while failure of the second gives \textup{(F2)}.

Now suppose that a component outside $P$ is nontrivial.
Its center together with one of its leaves gives a pendant path of length two.
If its root is congruent to $2$ modulo $3$, retaining only the center already gives \textup{(F1)}.
If its root is congruent to $1$ or $0$ modulo $3$, the diameter condition excludes the boundary positions, and the component gives \textup{(F3)} or \textup{(F4)}, respectively.
The converse follows immediately from the description of the sets $V_s$.
\end{proof}

\begin{lemma}
\label{lem:F-spectral}
If $T'$ is one of \textup{(F1)--(F4)}, then
$$ m_{T'}[0,1)=\left\lceil\frac{d+1}{3}\right\rceil+1. $$
\end{lemma}

\begin{proof}
For \textup{(F1)}, let the added leaf be adjacent to $v_{3a-1}$.
The subtree induced by
$$ \{v_{3a-3},v_{3a-2},v_{3a-1},v_{3a},v_{3a+1}\} $$
together with the added leaf is $T_6$.
Cut $v_{3a-4}v_{3a-3}$ when $a\ge2$, and cut $v_{3a+1}v_{3a+2}$ when $a\le\frac{d-4}{3}$.
The outer components, when present, are paths of diameters $3a-4$ and $d-3a-2$, both congruent to $2$ modulo $3$.
Their joining ends are nonzero by Lemma~\ref{lem:path-count}.
Repeated use of Lemma~\ref{lem:cut-additivity} and the value $m_{T_6}[0,1)=3$ gives
$$ m_{T'}[0,1)=(a-1)+3+\frac{d-3a-1}{3}=\left\lceil\frac{d+1}{3}\right\rceil+1, $$
where an absent path contributes zero.

For \textup{(F2)}, let the two leaves be attached to $v_{3a}$ and $v_{3b+1}$.
Assume first that $a<b$, and delete $v_{3a+1}v_{3a+2}$.
The component containing $v_0$ has diameter $3a+1$ and is an extremal caterpillar, so Lemma~\ref{lem:caterpillar-equality} gives $a+1$ eigenvalues less than $1$.
Its joining vertex $v_{3a+1}$ is nonzero by Lemma~\ref{lem:nonzero-tail}.
The component containing $v_d$ has diameter $d-3a-2\equiv2\pmod3$, and its remaining leaf is attached at a position congruent to $2$ modulo $3$ relative to its displayed diametral path.
Lemma~\ref{lem:wrong-leaf} gives $\frac{d-3a+2}{3}$ eigenvalues less than $1$.
Hence Lemma~\ref{lem:cut-additivity} yields $$m_{T'}[0,1)=\left\lceil\frac{d+1}{3}\right\rceil+1.$$

If $a=b$, delete $v_{3a-1}v_{3a}$.
The component containing $v_0$ is a path of diameter $3a-1$ and contributes $a$.
In the component containing $v_d$, the path beginning with the leaf adjacent to $v_{3a}$ is diametral, has diameter $d-3a+1$, and the other leaf is attached at a forbidden residue class.
Lemma~\ref{lem:wrong-leaf} gives $\frac{d-3a+5}{3}$ eigenvalues less than $1$.
The joining end of the component containing $v_0$ is nonzero, and therefore $$m_{T'}[0,1)=a+\frac{d-3a+5}{3}=\left\lceil\frac{d+1}{3}\right\rceil+1.$$

For \textup{(F3)}, write the pendant path as $v_{3a+1}uw$ and delete $v_{3a-1}v_{3a}$.
The component containing $v_0$ is a path of diameter $3a-1$ and contributes $a$.
Since this component is $P_{3a}$, its joining end $v_{3a-1}$ is nonzero for the eigenvalue $1$ by Lemma~\ref{lem:path-count}.
In the other component, $wuv_{3a+1}v_{3a+2}\cdots v_d$ is a diametral path of length $d-3a+1$, while $v_{3a}$ is a leaf attached at a forbidden position, and Lemma~\ref{lem:wrong-leaf} gives $\frac{d-3a+5}{3}$ eigenvalues less than $1$.
Hence Lemmas~\ref{lem:cut-additivity} and \ref{lem:wrong-leaf} give
$$ m_{T'}[0,1)=a+\frac{d-3a+5}{3}=\left\lceil\frac{d+1}{3}\right\rceil+1. $$

For \textup{(F4)}, write the pendant path as $v_{3a}uw$ and delete $v_{3a-1}v_{3a}$.
The component containing $v_0$ is a path of diameter $3a-1$, while the other component is the path $wuv_{3a}v_{3a+1}\cdots v_d$ of diameter $d-3a+2$.
The joining end of the component containing $v_0$ is nonzero, so Lemmas~\ref{lem:cut-additivity} and \ref{lem:path-count} give
$$ m_{T'}[0,1)=a+\frac{d-3a+5}{3}=\left\lceil\frac{d+1}{3}\right\rceil+1. $$
\end{proof}

\begin{corollary}
\label{cor:d-one}
Let $T$ have diameter $d\equiv1\pmod3$, and suppose every vertex of $T$ has distance at most $2$ from a diametral path $P$.
If $m_T[0,1)=\left\lceil\frac{d+1}{3}\right\rceil$, then $T$ is a caterpillar whose pendant vertices are attached only at vertices contained in a minimum dominating set of $P$.
\end{corollary}

\begin{proof}
Otherwise Lemma~\ref{lem:F-structure} gives a subtree $T'$ containing $P$ and having one of the forms \textup{(F1)--(F4)}.
Since $T'$ contains the whole diametral path $P$ and every component of $T-V(P)$ is a star with its center adjacent to $P$, the vertices of $T-T'$ can be restored successively as pendant vertices: for each component of $T-V(P)$ not contained in $T'$, first restore its center if necessary and then restore its leaves.
Lemmas~\ref{lem:pendant-monotonicity} and \ref{lem:F-spectral} give $m_T[0,1)\ge m_{T'}[0,1)=\left\lceil\frac{d+1}{3}\right\rceil+1$, a contradiction.
\end{proof}

\subsection{The case \texorpdfstring{$d\equiv0\pmod3$}{d congruent to 0 modulo 3}}

Let $d=3k$, where $k\ge2$, and regard $P$ as an ordered path.
For $1\le i\le d-1$, define
$$ \ell(i)=\begin{cases}0,&i\equiv1\pmod3,\\1,&i\equiv0\pmod3,\\2,&i\equiv2\pmod3.\end{cases} $$
For a leaf-only extremal caterpillar, the labels of the occupied positions are nondecreasing from left to right.
This is exactly the description of the positions where pendant vertices may be attached, given in the introduction.

Whether $P$ is labeled from left to right or from right to left, we consider the following structures:
\begin{description}
\item[\textup{(C1)}] Pendant vertices are attached to $v_i$ and $v_j$, where $1\le i<j\le d-1$ and $\ell(i)>\ell(j)$.
\item[\textup{(C2)}] A pendant path of length two is attached to $v_i$, where $2\le i\le d-2$ and $i\not\equiv0\pmod3$.
\item[\textup{(C3)}] A pendant path of length two is attached to $v_{3a}$ and an additional pendant vertex is attached to $v_{3b}$, where $1\le a\le b\le k-1$.
\end{description}

\begin{lemma}
\label{lem:C-structure}
Let $T$ have diameter $3k$, where $k\ge2$, and suppose every vertex of $T$ has distance at most $2$ from $P$.
Then $T$ belongs to class \textup{(i)} or \textup{(ii)} of Theorem~\ref{thm:classification-intro} if and only if $T$ contains none of \textup{(C1)--(C3)}, whether $P$ is labeled from left to right or from right to left.
\end{lemma}

\begin{proof}
If every component outside $P$ is a single pendant vertex, the attachment positions are contained in a minimum dominating set of $P$ precisely when their labels are nondecreasing.
Indeed, normalize a minimum dominating set as
$$ D=\{v_{a_1},\ldots,v_{a_{k+1}}\},\qquad a_1=1,\ a_{k+1}=3k-1, $$
replacing $v_0$ by $v_1$ and $v_{3k}$ by $v_{3k-1}$ if necessary, which removes no internal attachment vertex. Set $t_j=a_{j+1}-a_j$. Since $D$ dominates $P$, we have $1\le t_j\le3$, and
$$ \sum_{j=1}^{k}t_j=(3k-1)-1=3k-2,\qquad \text{so}\qquad \sum_{j=1}^{k}(3-t_j)=2. $$
Thus either one gap equals $1$ and the remaining gaps equal $3$, or two gaps equal $2$ and the remaining gaps equal $3$. In the former case the labels $\ell(a_j)$ change once, from $0$ to $2$; in the latter they increase from $0$ to $1$ and then to $2$; in both cases they are nondecreasing. Conversely, if the occupied positions have nondecreasing labels, then placing a gap of size $1$ where the label changes from $0$ to $2$, or two gaps of size $2$ where the label changes from $0$ to $1$ and from $1$ to $2$, and gaps of size $3$ elsewhere, yields a minimum dominating set of $P$ containing all of them.
Thus failure of class \textup{(i)} is exactly \textup{(C1)}.

Suppose now that $T-V(P)$ has a nontrivial component.
Selecting its center and one of its leaves produces a pendant path of length two.
If its root is not divisible by $3$, this is \textup{(C2)}.
Hence assume that its root is $v_{3a}$.
For $T$ to belong to class \textup{(ii)}, this must be the unique nontrivial component, and every ordinary pendant vertex must be attached in
$$ E_a=\{v_1,v_4,\ldots,v_{3a-2}\}\cup\{v_{3a+2},v_{3a+5},\ldots,v_{3k-1}\}. $$
A second nontrivial component rooted at a multiple of $3$, or an ordinary leaf at a multiple of $3$, gives \textup{(C3)} after a suitable choice of the labeling direction of $P$.
When only the center of the selected nontrivial component is retained, it becomes a pendant vertex adjacent to $v_{3a}$.
Thus an ordinary leaf to the left of $v_{3a}$ outside the first part of $E_a$, or to the right outside the second part, gives a label inversion with this pendant vertex and hence gives \textup{(C1)}.
These alternatives exhaust all failures of class \textup{(ii)}.
The converse is immediate from the two constructions in the theorem.
\end{proof}

\begin{lemma}
\label{lem:C-spectral}
If $T'$ is one of \textup{(C1)--(C3)}, then
$$ m_{T'}[0,1)\ge \left\lceil\frac{d+1}{3}\right\rceil+1. $$
\end{lemma}

\begin{proof}
For \textup{(C1)}, the possible label inversions are $(1,0)$, $(2,1)$ and $(2,0)$.
Reversing $P$ transforms $(2,1)$ into $(1,0)$, so it is enough to treat $(1,0)$ and $(2,0)$.

First suppose $i\equiv0\pmod3$ and $j\equiv1\pmod3$.
Delete $v_{i-1}v_i$.
The left component is a path of diameter $d_1=i-1\equiv2\pmod3$, and its joining end is nonzero.
In the other component, the path beginning with the leaf at $v_i$ is diametral and has diameter $d_2=d-i+1\equiv1\pmod3$.
Relative to this path, the leaf at $v_j$ is attached at a position congruent to $2$ modulo $3$, so it is an \textup{(F1)} structure.
Lemma~\ref{lem:F-spectral} gives $m_{T'_2}[0,1)\ge\lceil(d_2+1)/3\rceil+1$.
Since $d_1+d_2=d$, Lemma~\ref{lem:cut-additivity} gives
$$ m_{T'}[0,1)\ge\frac{d_1+1}{3}+\frac{d_2+2}{3}+1=\left\lceil\frac{d+1}{3}\right\rceil+1. $$

Now suppose $i\equiv2\pmod3$ and $j\equiv1\pmod3$.
Delete $v_{j-1}v_j$.
The component $T'_1$ containing $v_0$ has diameter $d_1=j-1\equiv0\pmod3$ and is an extremal caterpillar, since its unique attachment position $v_i$ is contained in a minimum dominating set of its displayed diametral path.
The other component $T'_2$ is a path of diameter $d_2=d-j+1\equiv0\pmod3$.
The vertex $v_i$ is the rightmost vertex of the displayed path having a pendant neighbour outside the displayed path, and $j-1-i\equiv1\pmod3$.
Thus the joining vertex $v_{j-1}$ is nonzero by Lemma~\ref{lem:nonzero-tail}.
Lemmas~\ref{lem:caterpillar-equality}, \ref{lem:path-count} and \ref{lem:cut-additivity} yield
$$ m_{T'}[0,1)=\left(\frac{d_1}{3}+1\right)+\left(\frac{d_2}{3}+1\right)=\left\lceil\frac{d+1}{3}\right\rceil+1. $$

For \textup{(C2)}, write the pendant path as $v_iuw$.
If $i=2$ or $i=d-2$, reverse $P$ if necessary and assume $i=2$.
Delete $v_3v_4$.
The left component is $T_6$, and the right component is a path of diameter $d-4\equiv2\pmod3$ whose joining end is nonzero.
Hence
$$ m_{T'}[0,1)=3+\frac{d-3}{3}=\left\lceil\frac{d+1}{3}\right\rceil+1. $$

Assume now that the root is not a boundary position.
After reversing $P$ if necessary, we may suppose $i\equiv1\pmod3$ and $4\le i\le d-5$.
Delete $v_{i-2}v_{i-1}$ and $v_{i+2}v_{i+3}$.
The middle component is $T_6$, and the two outer components are paths of diameters $i-2$ and $d-i-3$, both congruent to $2$ modulo $3$.
Their joining ends are nonzero, so repeated use of Lemma~\ref{lem:cut-additivity} gives
$$ m_{T'}[0,1)=\frac{i-1}{3}+3+\frac{d-i-2}{3}=\left\lceil\frac{d+1}{3}\right\rceil+1. $$

For \textup{(C3)}, write the pendant path as $v_{3a}uw$.
Suppose first that $a<b$, and delete $v_{3a+1}v_{3a+2}$.
In the left component $T'_1$, the path $wuv_{3a}v_{3a-1}\cdots v_0$ is diametral and has diameter $d_1=3a+2\equiv2\pmod3$, while $v_{3a+1}$ is a leaf attached at its third vertex.
Lemma~\ref{lem:wrong-leaf} gives $m_{T'_1}[0,1)=a+2$.
The right component $T'_2$ has diameter $d_2=d-3a-2\equiv1\pmod3$ and is an extremal caterpillar, since its unique attachment position $v_{3b}$ is contained in a minimum dominating set of its displayed diametral path.
Thus $m_{T'_2}[0,1)=\frac{d-3a}{3}$.
The attachment vertex $v_{3b}$ lies at distance $3(b-a)-2\equiv1\pmod3$ from the joining end $v_{3a+2}$.
Lemma~\ref{lem:nonzero-tail} shows that $v_{3a+2}$ is nonzero.
Therefore $$m_{T'}[0,1)=a+2+\frac{d-3a}{3}=\left\lceil\frac{d+1}{3}\right\rceil+1.$$
If $a=b$, delete $v_{3a-1}v_{3a}$.
The component containing $v_0$ is a path of diameter $3a-1$ and contributes $a$.
In the other component, $wuv_{3a}v_{3a+1}\cdots v_d$ is a diametral path of length $d-3a+2\equiv2\pmod3$, and the additional leaf at $v_{3a}$ is attached at its third vertex.
Lemma~\ref{lem:wrong-leaf} gives $\frac{d-3a}{3}+2$ eigenvalues less than $1$.
The joining end of the left path is nonzero, and hence $$m_{T'}[0,1)=a+\frac{d-3a}{3}+2=\left\lceil\frac{d+1}{3}\right\rceil+1.$$
\end{proof}

\begin{corollary}
\label{cor:d-zero}
Let $T$ have diameter $d\equiv0\pmod3$ with $d\ge6$, and suppose every vertex of $T$ has distance at most $2$ from a diametral path $P$.
If $m_T[0,1)=\left\lceil\frac{d+1}{3}\right\rceil$, then $T$ belongs to class \textup{(i)} or \textup{(ii)} of Theorem~\ref{thm:classification-intro}.
\end{corollary}

\begin{proof}
Otherwise Lemma~\ref{lem:C-structure} gives a subtree $T'$ containing $P$ and having one of the forms \textup{(C1)--(C3)}.
Since $T'$ contains the whole diametral path $P$ and every component of $T-V(P)$ is a star with its center adjacent to $P$, the vertices of $T-T'$ can be restored successively as pendant vertices: for each component of $T-V(P)$ not contained in $T'$, first restore its center if necessary and then restore its leaves.
Lemmas~\ref{lem:pendant-monotonicity} and \ref{lem:C-spectral} give $m_T[0,1)\ge m_{T'}[0,1)\ge \left\lceil\frac{d+1}{3}\right\rceil+1$, a contradiction.
\end{proof}

\section{Proof of Theorem 1.4}\label{sec:4}

We now prove the complete characterization stated in the introduction.

\begin{proof}[Proof of Theorem~\ref{thm:classification-intro}]
We first prove sufficiency.

Assume that $T$ satisfies (i), and let $D$ be the minimum dominating set of $P$ appearing in the statement.
Since the added pendant vertices are adjacent only to internal vertices of $P$, the path $P$ remains diametral.
Indeed, the distance from such a vertex adjacent to $v_i$ to either endpoint of $P$ is at most $d$, and the distance between two such vertices adjacent to $v_i$ and $v_j$ is at most $1+|i-j|+1\le 1+(d-2)+1=d$.
Thus the diameter is exactly $d$.

Since all added pendant vertices are adjacent to vertices of $D$, the set $D$ dominates all of $T$.
Hence $\gamma(T)\le |D|=\lceil(d+1)/3\rceil$.
By Lemma~\ref{lem:dom-lower}, $\gamma(T)\ge \lceil(d+1)/3\rceil$, so $\gamma(T)=\lceil(d+1)/3\rceil$.
Theorem~\ref{thm:HJT} and Lemma~\ref{lem:spectral-lower} give $\lceil(d+1)/3\rceil \le m_T[0,1) \le \gamma(T) = \lceil(d+1)/3\rceil$.
Therefore $m_T[0,1)=\lceil(d+1)/3\rceil$.

Now assume that $T$ satisfies (ii), and let $u$ be the center of the star $K_{1,t}$.
First we check that the path $P=v_0v_1\cdots v_{3k}$ is a diametral path.
Let $w$ be any vertex outside $P$, and let $v_i$ be the vertex of $P$ closest to $w$.
If $w$ is a pendant vertex adjacent directly to $P$, then $1\le i\le3k-1$, so $d_T(w,v_0)=i+1\le3k$ and $d_T(w,v_{3k})=3k-i+1\le3k$.
If $w=u$, then $i=3r$, and $d_T(u,v_0)=3r+1\le3k$ and $d_T(u,v_{3k})=3k-3r+1\le3k$.
If $w$ is a pendant vertex of $K_{1,t}$, then $d_T(w,v_0)=3r+2\le3k$ and $d_T(w,v_{3k})=3k-3r+2\le3k$, because $1\le r\le k-1$.

It remains to check distances between two vertices outside $P$.
Two pendant vertices adjacent directly to $P$ can only occur at indices between $1$ and $3k-1$, so their distance is at most $1+(3k-2)+1=3k$.
A pendant vertex of $K_{1,t}$ and a pendant vertex adjacent on the left have distance at most $2+(3r-1)+1=3r+2\le3k$, and a pendant vertex of $K_{1,t}$ and a pendant vertex adjacent on the right have distance at most $2+(3k-1-3r)+1=3k-3r+2\le3k$.
The same estimates with $2$ replaced by $1$ cover the center $u$.
Thus every pair of vertices is at distance at most $3k$, while $d_T(v_0,v_{3k})=3k$.
Hence the diameter is exactly $d=3k$.

Define
$$ D = \{v_1,v_4,\ldots,v_{3r-2}\} \cup \{u\} \cup \{v_{3r+2},v_{3r+5},\ldots,v_{3k-1}\}.$$
The vertices $v_1,v_4,\ldots,v_{3r-2}$ dominate the left part of the path up to $v_{3r-1}$.
The vertex $u$ dominates $v_{3r}$ and every vertex of $K_{1,t}$.
The vertices $v_{3r+2},v_{3r+5},\ldots,v_{3k-1}$ dominate the right part of the path from $v_{3r+1}$ to $v_{3k}$.
They also dominate all pendant vertices adjacent directly to $P$.
Thus $D$ dominates $T$.
Its size is $r+1+(k-r)=k+1 = \lceil(3k+1)/3\rceil$.
As above, Lemma~\ref{lem:dom-lower}, Theorem~\ref{thm:HJT}, and Lemma~\ref{lem:spectral-lower} imply $m_T[0,1)=\lceil(d+1)/3\rceil$.

We now prove necessity.
Assume $m_T[0,1)=\lceil(d+1)/3\rceil$.
If $d=0$, then $T\cong P_1$, which satisfies (i).
If $d=1$, then $T\cong P_2$, which also satisfies (i).
Hence assume $d\ge2$.
By Corollary~\ref{cor:distance-two}, there is a diametral path $P=v_0v_1\cdots v_d$ such that every vertex of $T$ has distance at most $2$ from $P$.

If $d\equiv2\pmod3$, the characterization of Guo, Xue and Liu~\cite{GXL} shows that $T$ is obtained from $P$ by adding pendant vertices only at $v_1,v_4,\ldots,v_{d-1}$.
Thus $T$ satisfies (i).

If $d\equiv1\pmod3$, Corollary~\ref{cor:d-one} shows that $T$ satisfies (i).

Finally, suppose $d\equiv0\pmod3$.
If $d=3$, then $T$ is a double star and its two central vertices form a minimum dominating set of a diametral path, so $T$ satisfies (i).
If $d\ge6$, Corollary~\ref{cor:d-zero} shows that $T$ satisfies (i) or (ii).

The proof is complete.
\end{proof}

\begin{corollary}
\label{cor:d-mod-2}
If $d\equiv2\pmod3$, then the trees attaining equality are exactly the caterpillars obtained from a diametral path $v_0v_1\cdots v_d$ by adding pendant vertices only at the vertices $v_1,v_4,\ldots,v_{d-1}$.
\end{corollary}

\begin{proof}
If $d\equiv2\pmod3$, then $d+1\equiv0\pmod3$.
The path $P_{d+1}$ has a unique minimum dominating set, namely $\{v_1,v_4,\ldots,v_{d-1}\}$.
Alternative (ii) cannot occur because it requires $d\equiv0\pmod3$.
The result follows from Theorem~\ref{thm:classification-intro}.
\end{proof}

\begin{corollary}
\label{cor:domination-equivalence}
For every tree $T$ of diameter $d$, $m_T[0,1) = \lceil(d+1)/3\rceil$ if and only if $\gamma(T) = \lceil(d+1)/3\rceil$.
\end{corollary}

\begin{proof}
If $\gamma(T) = \lceil(d+1)/3\rceil$, then Theorem~\ref{thm:HJT} and Lemma~\ref{lem:spectral-lower} imply $\lceil(d+1)/3\rceil \le m_T[0,1) \le \gamma(T) = \lceil(d+1)/3\rceil$.
Hence equality holds.

Conversely, if equality holds for $m_T[0,1)$, then Theorem~\ref{thm:classification-intro} shows that $T$ satisfies (i) or (ii).
In both cases, the sufficiency proof explicitly constructs a dominating set of size $\lceil(d+1)/3\rceil$.
By Lemma~\ref{lem:dom-lower}, this dominating set is minimum.
\end{proof}

\section{Proof of Theorem 1.5}\label{sec:5}

We now count the equality trees directly from the structural characterization in Theorem~\ref{thm:classification-intro}.
The proof is a direct algebraic enumeration based on the binomial theorem and the standard count of nonnegative integer solutions.

\begin{lemma}
\label{lem:counting}
For every $n\ge3$,
$$ |\calT_n^*|\le(n+2)2^{n-3}. $$
\end{lemma}

\begin{proof}
By Theorem~\ref{thm:classification-intro}, every tree in $\calT_n^*$ belongs to class \textup{(i)} or class \textup{(ii)}.

We first count class \textup{(i)}.
Fix the diameter $d$ and choose an orientation $P=v_0v_1\cdots v_d$ of the diametral path occurring in the characterization.
Let $a_i$ be the number of pendant vertices attached to $v_i$.
For an upper bound, we disregard the restrictions imposed by the minimum dominating set and allow $a_i\ge0$ at every path vertex.
Then
$$ a_0+a_1+\cdots+a_d=n-d-1. $$
Therefore, the number of such vectors is
$ \binom{n-1}{d}. $
Every tree in class \textup{(i)} has at least one such representation, while different vectors or orientations may produce the same unlabeled tree.
Consequently, the number of trees in class \textup{(i)} is at most
$$ \sum_{d=0}^{n-1}\binom{n-1}{d}=2^{n-1}. $$

We next count class \textup{(ii)}.
Again fix an oriented path $P=v_0v_1\cdots v_d$, and deliberately disregard the congruence and position restrictions in the characterization.
There are at most $d+1$ choices for the path vertex to which the star center is joined.
Let $s\ge2$ be the order of the attached star component, including its center, and let $a_i\ge0$ be the number of pendant vertices attached directly to $v_i$.
The order condition becomes
$$ s-2+a_0+a_1+\cdots+a_d=n-d-3. $$
For fixed $d$, the number of nonnegative solutions is $\binom{n-2}{d+1}$.
Hence the number of trees in class \textup{(ii)} is at most
$$ \sum_{d=0}^{n-3}(d+1)\binom{n-2}{d+1}=(n-2)2^{n-3}. $$

Adding the two estimates gives
$$ |\calT_n^*|\le2^{n-1}+(n-2)2^{n-3}=(n+2)2^{n-3}. $$
\end{proof}

\begin{lemma}[{\cite{Otter}}]
\label{lem:Otter}
Let $\calT_n$ denote the set of unlabeled trees on $n$ vertices.
Then
$$ |\calT_n| \sim C_0\alpha^n n^{-5/2}, $$
where $C_0>0$ and $\alpha\approx 2.955765\ldots>\frac52$ is Otter's tree constant.
\end{lemma}

Combining Lemma~\ref{lem:counting} with Lemma~\ref{lem:Otter} immediately implies the asymptotic result.

\begin{proof}[Proof of Theorem~\ref{thm:asymptotic-intro}]
Lemma~\ref{lem:counting} gives $|\calT_n^*|=O(n2^n)$.
On the other hand, Lemma~\ref{lem:Otter} gives $|\calT_n| \sim C_0\alpha^n n^{-5/2}$, where $C_0>0$ and $\alpha\approx 2.955765\ldots>\frac52$.
Hence
$$ \frac{|\calT_n^*|}{|\calT_n|} = O\!\left(n^{7/2}\left(\frac{2}{\alpha}\right)^n\right) \xrightarrow[n\to\infty]{}0. $$
This proves the theorem.
\end{proof}

\begin{corollary}
\label{cor:almost-all-plus-one}
Almost all trees satisfy $m_T[0,1) \ge \lceil(d(T)+1)/3\rceil+1$.
Equivalently,
$$ \lim_{n\to\infty} \frac{ \left| \left\{ T\in\calT_n: m_T[0,1)\ge \left\lceil\frac{d(T)+1}{3}\right\rceil+1 \right\} \right| }{|\calT_n|} = 1.$$
\end{corollary}

\begin{proof}
By Lemma~\ref{lem:spectral-lower}, every tree $T$ satisfies $m_T[0,1)\ge \lceil(d(T)+1)/3\rceil$.
Thus the complement of the family in the numerator is precisely $\calT_n^*$.
The result follows from Theorem~\ref{thm:asymptotic-intro}.
\end{proof}

\section*{Statement on AI usage}

The proof strategy in Section~\ref{sec:5} was developed with the assistance of GPT-5.4. By contrast, both the proof ideas and the detailed arguments in Sections~\ref{sec:1}--\ref{sec:4} were developed entirely by the authors, without any use of AI.

\end{document}